\documentclass[12pt]{amsart}
\usepackage{amssymb,slashed, color,array}
\usepackage{amsmath,amsfonts,amsthm,euscript,mathrsfs,multirow}
\usepackage[all]{xy}
\usepackage[english]{babel}
\usepackage{epsf}
\usepackage{graphicx}
\csname @addtoreset\endcsname{equation}{section}

\newtheorem{theorem}{Theorem}[section]
\newtheorem{lemma}[theorem]{Lemma}
\newtheorem{corol}[theorem]{Corollary}

\theoremstyle{definition}
\newtheorem{definition}[theorem]{Definition}
\theoremstyle{remark}

\newenvironment{notation and conventions}{\textbf{Notation and conventions.}}{ }
\DeclareFontFamily{U}{rsf}{} \DeclareFontShape{U}{rsf}{m}{n}{ <5> <6> rsfs5 <7> <8> <9> rsfs7 <10-> rsfs10}{}
\DeclareMathAlphabet\Scr{U}{rsf}{m}{n}
\definecolor{pink}{rgb}{1,0,1}

  \usepackage[pdftex]{hyperref}

\begin{document}

\begin{titlepage}
\begin{center}
\baselineskip=14pt{\LARGE
 On Hirzebruch invariants of elliptic fibrations\\
}
\vspace{3 cm}
{\large  James Fullwood$^{\clubsuit,}$ , Mark van Hoeij$^\clubsuit$
  } \\

\vspace{.6cm}

${}^\clubsuit$Mathematics Department, Florida State University, Tallahassee, FL 32306, U.S.A.\\

\end{center}

\vspace{1cm}
\begin{center}

{\bf Abstract}
\vspace{.3 cm}
\end{center}

{\small

We compute all Hirzebruch invariants $\chi_q$ for $D_5$, $E_6$, $E_7$ and $E_8$ elliptic fibrations of every dimension. A single generating series $\chi(t,y)$ is produced for each family of fibrations such that the coefficient of $t^{k}y^{q}$ encodes $\chi_q$ over a base of dimension $k$, solely in terms of invariants of the base of the fibration.
\vfill

$^\clubsuit$Email:\quad {\tt jfullwoo at math.fsu.edu, hoeij at math.fsu.edu}
}

\end{titlepage}
\addtocounter{page}{1}
 \tableofcontents{}
\newpage

\section{Introduction}\label{intro}
The prospect of realizing realistic particle physics (such as the Standard Model) in a regime of string theory coined ``F-theory" by its originator Cumrun Vafa \cite{Vafa:1996xn}\cite{Denef les houches} has provided a source of attraction for string theorists and their mathematician counterparts to the study of elliptic fibrations. F-theory was first formulated as a non-perturbative description of Type-IIB string theory on a complex $n$-fold $B$ with an $SL_{2}(\mathbb{Z})$-invariant complex scalar field\footnote{Technically this is an abuse of language, as it is the whole type-IIB theory which is $SL_{2}(\mathbb{Z})$-invariant.} known to physicists as the axio-dilaton field. Exercising a string theorist's natural penchant for algebro-geometric descriptions of nature, Vafa formulated a geometrization of the $SL_{2}(\mathbb{Z})$-invariance of the axio-dilaton via a Calabi-Yau elliptic fibration over the type-IIB $n$-fold $B$ (which describes a theory in 10-2$n$ real space-time dimensions), interpreting the axio-dilaton as the complex structure modulus of an elliptic curve. Not only is this formulation of non-perturbative type-IIB string theory aesthetically pleasing from a purely geometric perspective, the physical theory has attractive features such as providing promising avenues for moduli-stabilization and potential realization of GUT gauge groups which project by definition to the Standard Model gauge group at lower non-supersymmetric energy levels. To realize the elliptic fibration explicitly, physicists have primarily focused on a Weierstrass fibration, i.e., a hypersurface in a $\mathbb{P}^2$-bundle over the Type-IIB base $B$ which in its reduced form is defined as the zero-scheme associated with the locus

\[
y^2z=x^3+fxz^2+gz^3,
\]
where $f$ and $g$ are sections of appropriate tensor powers of a line bundle $\mathscr{L}$ on $B$. As in the theory of curves, every smooth elliptic fibration is birational to a (possibly singular) fibration in Weierstrass form, often referred to in the physics literature as an $E_8$ elliptic fibration. But crucial to the physical theory associated with an elliptic fibration are the singular fibers of the fibration, as the singular fibers encode the structure of gauge theories associated with D-branes wrapping components of the discriminant locus over which they appear. And since singular fibers of a fibration are in general not preserved under a birational transformation, elliptic fibrations not in Weierstrass form enjoy their own physical relevance.

Motivated by tadpole cancellation in F-theory, in \cite{SVW} Sethi,Vafa and Witten derived a formula for the Euler characteristic of an elliptically fibered Calabi-Yau fourfold in Weierstrass form solely in terms of the Chern classes of the base of the fibration. Similar formulas for fibrations not in Weierstrass form were derived by Klemm, Lian, Roan and Yau in \cite{KLRY}. It was later shown by Aluffi and Esole in \cite{AE1}\cite{AE2} that these formulas are all numerical avatars of more general Chern class identities which hold not only without any Calabi-Yau hypothesis but over a base of arbitrary dimension. In this note we consider four families of fibrations $\varphi:Y\to B$ which are known to the physics community as $D_5$, $E_6$, $E_7$ and $E_8$ elliptic fibrations from a purely mathematical perspective (i.e., with no Calabi-Yau hypothesis or restrictions on the dimension of the base), and pursue similar formulas not for the Chern classes of a given fibration $Y$ but for its Hirzebruch invariants (or arithmetic genera)

\[
\chi_{q}(Y):=\int \text{ch}(\Omega_{Y}^{q})\text{td}(Y),
\]
where $\text{ch}(\Omega_{Y}^{q})$ denotes the Chern character of the $q$th exterior power of the cotangent bundle of $Y$ and $\text{td}(Y):=\text{td}(TY)\cap [Y]$, i.e., the Todd class of the tangent bundle of $Y$ acting on the fundamental class of $Y$.

As integrals are invariant under proper pushforwards of the integrand, we relate $\chi_{q}(Y)$ to invariants of the base by pushing forward $\text{ch}(\Omega_{Y}^{q})\text{td}(Y)$ via $\varphi_{*}$ (the pushforward map associated with $\varphi:Y\to B$). By the celebrated Hirzebruch-Riemann-Roch theorem (later generalized by Grothendieck),
\[
\chi_{q}(Y)=\sum_{i} (-1)^i \text{dim} H^{i}(Y,\Omega_{Y}^q)=h^{q,0}-h^{q,1}+\cdots +(-1)^{q}h^{q,q},
\]
thus Hirzebruch invariants yield linear relations among the Hodge numbers of $Y$. For a general smooth complex projective variety $X$ of fixed dimension, the standard approach to computing Hirzebruch invariants of $X$ is to encode them in a generating series

\[
\tag{$\dagger$}
\chi(y)=\displaystyle\sum_{q} \chi_{q}y^q=\int_{X} \prod_{i=1}^{\text{dim}(X)}(1+ye^{-\lambda_{i}})\frac{\lambda_i}{1-e^{-\lambda_i}},
\]
where the $\lambda_{i}$s are the Chern roots of the tangent bundle of $X$. Given an elliptic fibration $\varphi:Y\to B$ of type $D_5$, $E_6$, $E_7$ or $E_8$ over a base $B$ of \emph{arbitrary dimension}, what we achieve in this note is a single generating series $\chi(t,y)$ for each family where the coefficient of $t^{k}y^{q}$ encodes $\chi_q$ for the given family of elliptic fibrations over a base of dimension $k$, solely in terms of Chern classes of the base and the first Chern class of a line bundle $\mathscr{L}\to B$ (i.e., one can see the Hirzebruch invariants of the fibration as \emph{functions} of invariants of the base).

Let $B$ be a smooth compact complex projective variety of arbitrary dimension endowed with a line bundle $\mathscr{L}\to B$. The elliptic fibrations we consider will all be subvarieties of an ambient projective bundle $\mathbb{P}(\mathscr{E})\to B$ (each fibration will be precisely defined in \S\ref{def}), where $\mathscr{E}$ is a vector bundle over $B$ that is constructed by taking direct sums of tensor powers of $\mathscr{L}$. Before stating the main result of this note, let us make the following definitions:

Let $X$ be a smooth variety. We define the \emph{Hirzebruch series} of $X$ to be

\[
\mathscr{H}_{y}(X):=\mathscr{H}_{0}(X)+\mathscr{H}_{1}(X)y+\mathscr{H}_{2}(X)y^2+\cdots,
\]
\\
where $\mathscr{H}_{q}(X):=\text{ch}(\Omega_{X}^{q})\text{td}(X)$ is the $q$th \emph{Hirzebruch characteristic class} of $X$. Then given a proper morphism $\varphi:X\to B$ we define $\varphi_{*}(\mathscr{H}_{y}(X))$ in the obvious way:

\[
\varphi_{*}(\mathscr{H}_{y}(X)):=\varphi_{*}(\mathscr{H}_{0}(X))+\varphi_{*}(\mathscr{H}_{1}(X))y+\varphi_{*}(\mathscr{H}_{2}(X))y^2+\cdots,
\]
\\
where $\varphi_{*}$ is the proper pushforward associated with the morphism $\varphi$. Our main result is the following

\begin{theorem}\label{Mainresult}
Let $\varphi:Y\to B$ be an elliptic fibration of type $D_5$, $E_6$, $E_7$ or $E_8$ and let $U=e^{-c_1(\mathscr{L})}$. Then

\[
\varphi_{*}(\mathscr{H}_{y}(Y))=Q\cdot \mathscr{H}_{y}(B),
\]
where
\[
Q=
\begin{cases}
4-y + \frac{(y+1)(yU -3)}{(yU^2+1)} - \frac{U(y+1)^2}{(yU^2+1)^2} \hspace{0.25in} \text{for $Y$ a $D_5$ fibration} \\
3-y + \frac{(y+1)(yU^2-U-2)}{(yU^3+1)} \hspace{0.75in} \text{for $Y$ an $E_6$ fibration} \\
2-y + \frac{(y+1)(yU^3-U-1)}{(yU^4+1)} \hspace{0.75in} \text{for $Y$ an $E_7$ fibration} \\
1-y + \frac{(y+1)(yU^5-U-0)}{(yU^6+1)} \hspace{0.75in} \text{for $Y$ an $E_8$ fibration} \\
\end{cases}
\]
\end{theorem}

The proof of this result is considerably streamlined via the use of \emph{Chern-ext characters}, which we first introduce and define in \S\ref{Chernext}. We note that the numbers 4, 3, 2 and 1 in the expressions for $Q$ coincide with the number of sections of the given fibration. Reading off the coefficient of $y^q$ in $Q\cdot \mathscr{H}_{y}(B)$ immediately yields

\[
\chi_{q}(Y)=\int_{B} \left(P_{0}\mathscr{H}_{q}(B)+P_{1}\mathscr{H}_{q-1}(B)+\cdots +P_{q}{\rm td}(B)\right),
\]
where the $P_{i}$s are polynomials in $U=e^{-c_{1}(\mathscr{L})}$ (we list all $P_{i}$s in \S\ref{proof1}). Furthermore,
from Theorem~\ref{Mainresult} we derive generating series $\chi(t,y)$ for each family of fibrations, where the coefficient of $t^{k}y^q$ encodes $\chi_{q}$ for the given family of fibrations over a base of dimension $k$. Before unveiling the $\chi(t,y)$ we need the following definitions:
\\

Let $R$ be a commutative ring with unity. For two series $f(t)=a_0+a_1t+a_2t^2+\cdots$ and $g(t)=b_0+b_1t+b_2t^2+\cdots$ in $R[[t]]$, we recall the \emph{Hadamard product} of $f$ and $g$ is defined to be

\[
f\odot g:=a_0b_0+a_1b_1t+a_2b_2t^2+\cdots.
\]
\\
Furthermore, let $[t^d]:R[[t]] \rightarrow R$ denote the map given by $[t^d][f] := a_d$ . Now let $X$ be a smooth projective variety of dimension $d$, let $g=(1+ye^{-t})\frac{t}{1-e^{-t}}$ and let $f=\ln(g)$. As a consequence of Lemma~\ref{lemma1.1} we show

\[
\mathscr{H}_{y}(X)=(1+y)^d \cdot [t^d][{\rm exp}\left( f \odot (-tC'/C)   \right) ],
\]
\\
where $C=1-c_1t+c_2t^2-c_3t^3+\cdots \in R[[t]]$ with $R=\mathbb{Z}[c_1,c_2,\ldots,]$ \footnote{As $ f \odot (-tC'/C)$ is independent of $d$, the $c_{i}$s appearing in the definition of $C$ are countably many formal variables which acquire their familiar meaning as Chern classes of the tangent bundle of $X$ in concrete examples.}. As $\varphi_{*}(\mathscr{H}_{y}(Y))=Q\cdot \mathscr{H}_{y}(B)$ by Theorem~\ref{Mainresult}, the generating series $\chi(t,y)$ is constructed by replacing $U=e^{-c_{1}(\mathscr{L})}$ by $U_{t}=e^{-c_{1}(\mathscr{L})t}$ in $Q$ and then constructing a series in $R[[t]]$ with $R=\mathbb{Z}[c_1,c_2,\ldots]$ such that the coefficient of $t^d$ is precisely $(1+y)^d \cdot [t^d]{\rm exp}\left( f \odot (-tC'/C)   \right)$, which we interpret as the Hirzebruch series for a base of dimension $d$. All Hirzebruch invariants $\chi_q$ for $D_5$, $E_6$, $E_7$ and $E_8$ fibrations of all dimensions are then contained in the following

\begin{corol}\label{genseries}
Let $Q_{t}=Q(e^{-c_1(\mathscr{L})t})$, where $Q$ is defined as in Theorem~\ref{Mainresult} and let
\[
\overset \sim \chi(t,y)=Q_{t}\cdot \exp\left(\ln\left((1+ye^{-t})\frac{t}{1-e^{-t}}\right)\odot \frac{-tC'}{C}\right),
\]
where $C=1-c_{1}t+c_{2}t^2-\cdots \in R[[t]]$ with $R=\mathbb{Z}[c_1,c_2,\ldots]$ and $C'=\frac{d}{dt}C$. Then

\[
\chi(t,y):=\overset \sim \chi(t(1+y),y)
\]
\\
is a generating series for Hirzebruch invariants of $D_5$, $E_6$, $E_7$ and $E_8$ fibrations as the definition of $Q$ varies according to Theorem~\ref{Mainresult}, i.e., the integral of the coefficient of $t^{k}y^{q}$ over a base of dimension $k$ is precisely $\chi_q$ for the given family of fibrations.
\end{corol}

As an illustration, the coefficient of $t^{4}y^{2}$ in $\chi(t,y)$ in the $E_6$ case is $-\frac{L}{12}(1729L^3-524c_{1}L^2+(-17c_{1}^2+193c_2)L+5c_{1}c_{2}-66c_3)$, thus $\chi_{2}$ of an $E_6$ fibration over a base $B$ of dimension four is (for a computer implementation see \cite{URL})

\[
\int_{B} -\frac{L}{12}(1729L^3-524c_{1}(B)L^2+(-17c_{1}(B)^2+193c_{2}(B))L+5c_{1}(B)c_{2}(B)-66c_{3}(B)).
\]

\section{The fibrations under consideration}\label{def}

We now formally define the objects under consideration, namely $D_5$, $E_6$, $E_7$ and $E_8$ elliptic fibrations (the names and definitions we use are all lifted from the physics literature \cite{KLRY}). We work over $\mathbb{C}$ though everything we say is equally valid over an algebraically closed field of characteristic zero. All fibrations are constructed by taking equations of classical elliptic curves and promoting their coefficients from scalars (or sections of line bundles over a point) to sections of line bundles over some smooth positive dimensional base variety $B$. As such, let $B$ be some smooth compact complex projective variety of arbitrary dimension endowed with (a suitably ample) line bundle  $\mathscr{L}\to B$. This will be the base assumption in each of the $D_5$, $E_6$, $E_7$ and $E_8$ cases.

Now let $\mathscr{E}=\mathscr{O}\oplus \mathscr{L}\oplus \mathscr{L}\oplus \mathscr{L}$. A \emph{$D_5$ elliptic fibration} $Y_{D_5}$ is defined to be a smooth complete intersection in $\mathbb{P}(\mathscr{E})$ (here we take the projective bundle of \emph{lines} in $\mathbb{P}(\mathscr{E})$) associated with the locus

\[
Y_{D_5}:
\begin{cases}
x^2-y^2-z(az+cw)=0\\
w^2-x^2-z(dz+ex+fy)=0,
\end{cases}
\]
where $z$ is a section of $\mathscr{O}(1)$ (the dual of the tautological line bundle on $\mathbb{P}(\mathscr{E})$), and $x$, $y$ and $w$ are sections of $\mathscr{O}(1)\otimes \pi^{*}\mathscr{L}$, where $\pi$ is the projection $\pi:\mathbb{P}(\mathscr{E})\to B$. Then we take $a$, $c$, $d$, $e$ and $f$ to be sections of (minimal) appropriate tensor powers of $\pi^{*}\mathscr{L}$ such that each of the defining equations for $Y_{D_5}$ is a well defined section of a line bundle on $\mathbb{P}(\mathscr{E})$. Taking $a$ and $d$ to be sections of $\pi^{*}\mathscr{L}^2$ and $c$, $e$ and $f$ to be sections of $\pi^{*}\mathscr{L}$ then defines $Y_{D_5}$ as a variety of class $(2H+2L)^2\in A^{*}\mathbb{P}(\mathscr{E})$, where $H:=c_1(\mathscr{O}(1))$ and $L:=c_1(\pi^{*}\mathscr{L})$. Such a locus naturally determines an elliptic fibration $\varphi:Y_{D_5}\to B$, with generic fiber an elliptic curve in $\mathbb{P}^3$. Such fibrations contain fibers not on the list of Kodaira\footnote{More precisely, over the locus $a=c=d=e=f=0$ in the base the fibers consist of four $\mathbb{P}^{1}$s meeting at a point.} and were studied extensively in \cite{EFY} from both a mathematical and physical perspective. We note that as our results are topological in nature, they depend only on the class $[Y]\in A^{*}\mathbb{P}(\mathscr{E})$ of the given fibration, i.e., the explicit equations which define the fibration are essentially irrelevant. The equations are included for concreteness and to not completely sever ourselves from the physical theories with which they are associated. The definitions of $E_6$, $E_7$ and $E_8$ fibrations are summarized in the following table:
\\
\begin{center}
\begin{tabular}{|c|c|c|c|}
\hline
 & equation & ambient projective bundle & class in $A^{*}\mathbb{P}(\mathscr{E})$  \\
\hline
$E_6$ & $x^3+y^3=dxyz+exz^2+fyz^2+gz^3$ & $\mathbb{P}(\mathscr{O}\oplus \pi^{*}\mathscr{L}\oplus \pi^{*}\mathscr{L})$ & $3H+3L$ \\
\hline
$E_7$ & $y^2=x^4+ex^2z^2+fxz^3+gz^4$ & $\mathbb{P}_{1,1,2}(\mathscr{O}\oplus \pi^{*}\mathscr{L}\oplus \pi^{*}\mathscr{L}^2)$ & $4H+4L$ \\
\hline
$E_8$ & $y^{2}z=x^3+fxz^2+gz^3$ & $\mathbb{P}(\mathscr{O}\oplus \pi^{*}\mathscr{L}^{2}\oplus \pi^{*}\mathscr{L}^{3})$ & $3H+6L$ \\
\hline
\end{tabular}
\\
\end{center}

The coefficients of each fibration are chosen to be suitably generic sections of tensor powers of $\pi^{*}\mathscr{L}$ such that the fibration is a smooth divisor of the indicated class in $A^{*}\mathbb{P}(\mathscr{E})$. We point out that the total space of an  $E_7$ fibration is defined as a hypersurface in a \emph{weighted} projective bundle, which is isomorphic to a bundle of quadric cones. To avoid dealing with any singularities of the ambient projective bundle while performing intersection theoretic computations, we embed $\mathbb{P}_{1,1,2}(\mathscr{E})$ in a $\mathbb{P}^3$-bundle and then realize the total space of the $E_7$ fibration as a complete intersection of the image of $\mathbb{P}_{1,1,2}(\mathscr{E})$ via its embedding in the $\mathbb{P}^3$-bundle with another hypersurface.

\section{A motivating example}
Let $B$ be a non-singular compact complex algebraic variety of arbitrary dimension endowed with a (suitably ample) line bundle $\mathscr{L}$. We recall the definition of an $E_8$ elliptic fibration, i.e., a surjective proper morphism $\varphi:Y\to B$, whose total space $Y$ is realized as a hypersurface of class $3H+6L$ in the the Chow ring of a projective bundle $\mathbb{P}(\mathscr{E})\overset \pi \to B$, where $\mathscr{E}=\mathscr{O} \oplus \mathscr{L}^2 \oplus \mathscr{L}^3$, $H:=c_1(\mathscr{O}(1))$ and $L$ we non-reluctantly use to denote both $c_1(\mathscr{L})$ and $\pi^{*}c_1(\mathscr{L})$. As one can show that $\varphi_{*}c(Y)=\frac{12L}{1+6L}\cdot c(B)$ \cite{AE1}, we exploit the fact that $\int_{Y}c(Y)=\int_{B}\varphi_{*}c(Y)$ and compute the topological Euler characteristic $\chi(Y)$ by integrating the coefficient of $t^{dim(B)}$ in the formal series $\chi^{dim(B)}(t)$, where we define $\chi^{N}(t)$ for general $N\in \mathbb{N}$ to be \footnote{Here and throughout, terms not expanded in a series such as $\frac{12Lt}{1+6Lt}$ in the series above are a shorthand for their associated series expansions about $t$.}

\[
\chi^{N}(t):=\frac{12Lt}{1+6Lt} \cdot(1+c_1t+c_2t^2+ \cdots +c_{N}t^{N}).
\]
As the series $\chi^{n}(t)$ and $\chi^{n+1}(t)$ are identical up to order $n$ (for any $n$), we notice that the formal series
\[
\chi^{E_8}(t)=\frac{12Lt}{1+6Lt} \cdot(1+c_1t+c_2t^2+ \cdots +c_mt^m+ \cdots)
\]
serves as a generating series for the topological Euler characteristic for $Y$ of all possible dimensions (i.e., the coefficient of $t^k$  encodes the Euler characteristic of $Y$ over a base of dimension $k$, solely in terms of $L$ and Chern classes of $B$). The $c_is$ are then temporarily formal objects (as their subscripts tend towards infinity), acquiring their familiar meaning whence integrated upon. For example, over a base of dimension $3$ the coefficient of $t^3$ in $\chi^{E_8}(t)$ is $12L(c_2-6Lc_1+36L^2)$, so $\chi(Y)$ over a base $B$ of dimension $3$ is

\[
\int_{B} 12L(c_2(B)-6Lc_1(B)+36L^2).
\]

Though admittingly these observations are all rather trivial, what is striking is that the pushforward $\varphi_{*}c(Y)$ is manifestly independent of the base of the fibration, i.e., over a base of dimension $k$ the actual class in $A^{*}B$ associated with $\varphi_{*}c(Y)$ is obtained by truncating a formal powers eries at order $k$. As such, we deem this a ``motivating example" as it exhibits general features which can be abstracted to cases of other invariants of elliptic fibrations of the form $\int \alpha$. In particular, the first step at arriving at such base independent expressions for $\varphi_{*}c(Y)$ is deriving a factorization of $c(Y)$ (which plays the role of $\alpha$ in our more general considerations) as $g(L,H)\cdot \pi^{*}c(B)$, where $g$ is a rational expression depending only on $L$ and $H$ (as defined above). Once we have such a factorization, we get that

\[
\varphi_{*}c(Y)=\pi_{*}(g(L,H))c(B)=\frac{12L}{1+6L}c(B),
\]
an expression which depends in no way on the dimension of $B$. Essential to the base independence of the formula above is a pushforward formula which computes $\pi_{*}(g(L,H))$ in terms of a rational expression in $L$ whose associated series is truncated at the dimension of the base to obtain the given class. Such a pushforward formula was recently obtained in \cite{James}. More generally, for a given invariant of the form $I_{Y}=\int \alpha$ for some subvariety $Y$ of a projective bundle, we seek an analogous factorization  $\alpha=g(L,H)\pi^{*}(I_{B})$, where the class associated with $g$ is obtained by truncating a formal series in $L$ and $H$ at the dimension of $Y$. Then by applying the pushforward formula of \cite{James} to $g(L,H)$ we immediately arrive at a base independent expression for the pushforward of $\alpha$, which lends itself naturally to a generating series which encodes the invariant $I_{Y}$ for $Y$ over bases of arbitrary dimension. In what follows we successfully carry out this program for Hirzebruch invariants of $D_5$, $E_6$, $E_7$ and $E_8$ elliptic fibrations.

\section{The Chern-ext character}\label{Chernext}

We now define a series which plays a key role in our analysis:\\

\begin{definition}
Let $\mathscr{E}$ be a vector bundle, then we define the \emph{Chern-ext character} of $\mathscr{E}$ to be

\[
\text{ch}_{ext}(\mathscr{E}):=1+\text{ch}(\mathscr{E})y+\text{ch}(\Lambda^{2}\mathscr{E})y^2+\cdots.
\]
\end{definition}
We note that for any commutative ring $R$ with unity, any element of the form $f(y)=1+a_1y+a_2y^2+\cdots$ is a unit in $R[[y]]$, thus $\frac{1}{f(y)}$ is well defined. If $\mathscr{E}$ is of rank $r$ with Chern roots $(\lambda_{1},\lambda_{2},\cdots,\lambda_{r})$, then $\text{ch}_{ext}(\mathscr{E})$ is a polynomial in $y$ of degree $r$ which factors as

\[
\text{ch}_{ext}(\mathscr{E})=\prod_{i=1}^{r}(1+y\cdot \exp(\lambda_{i})).
\]
However, due to the dimension-independent nature of our results we prefer to think of $\text{ch}_{ext}$ as a series, which also makes more evident the invertability of the Chern-ext character. Furthermore, given a smooth projective variety $X$ we note that the Hirzebruch series of $X$ is simply

\[
\mathscr{H}_{y}(X)=\text{ch}_{ext}(\Omega_{X})\text{td}(X).
\]

Armed with such a series, we now prove the following

\begin{lemma}\label{L}
Let
\[
0\to \mathscr{A} \to \mathscr{B} \to \mathscr{C} \to 0
\]
be an exact sequence of vector bundles. Then

\[
{\rm ch}_{ext}(\mathscr{B})={\rm ch}_{ext}(\mathscr{A})\cdot {\rm ch}_{ext}(\mathscr{C})
\]
\end{lemma}

\begin{proof}
>From the $\lambda$-ring identity (\cite{FL}, pg.2)

\[
\lambda^{p}(x+y)=\displaystyle\sum_{i=0}^{p}\lambda^{i}(x)\lambda^{p-i}(y),
\]
along with nice properties of the Chern character (\cite{fulton}, example 3.2.3), we get

\[
\text{ch}(\Lambda^{p}\mathscr{B})=\displaystyle\sum_{i=0}^{p}\text{ch}(\Lambda^{i}\mathscr{A})\cdot \text{ch}(\Lambda^{p-i}\mathscr{C}).
\]
The lemma immediately follows.
\end{proof}

\section{The proof}
\subsection{Proof of main result}\label{proof1}
Let $\varphi:Y\to B$ be a $D_5$, $E_6$, $E_7$ or $E_8$ elliptic fibration as defined in \S\ref{def} and let $N$ denote both the normal bundle of $Y$ in $\mathbb{P}(\mathscr{E})$ and the bundle on $\mathbb{P}(\mathscr{E})$ which restricts to it. Using the exact sequences (we use a superscript ``$\vee$" to denote duals)

\begin{eqnarray*}
     & 0\to N^{\vee}\to i^{*}\Omega_{\mathbb{P}(\mathscr{E})} \to \Omega_Y \to 0 \\
     & 0\to \pi^{*}\Omega_{B}\to \Omega_{\mathbb{P}(\mathscr{E})}\to \Omega_{\mathbb(\mathscr{E})/B}\to 0 \\
     & 0\to \Omega_{\mathbb{P}(\mathscr{E})/B} \to (\pi^{*}\mathscr{E}\otimes \mathscr{O}(1))^{\vee}\to \mathscr{O}_{\mathbb{P}(\mathscr{E})} \to 0,  \\
\end{eqnarray*}
along with Lemma~\ref{L} we get

\[
i_{*}\left(\text{ch}_{ext}(\Omega_{Y})\text{td}(Y)\right)=\left(\frac{\text{ch}_{ext}(\mathscr{F}^{\vee})}{\text{ch}_{ext}(N^{\vee})(1+y)}\cdot \alpha\right)\pi^{*}\left(\text{ch}_{ext}(\Omega_{B})\text{td}(B)\right),
\]
\\
where $\mathscr{F}$ we use to denote $\pi^{*}\mathscr{E}\otimes \mathscr{O}(1)$, $\alpha=\frac{\mathscr{F}}{\text{td}(N)}\cap [Y]$, $i:Y\hookrightarrow \mathbb{P}(\mathscr{E})$ is the inclusion and we use the fact that $\text{ch}_{ext}(\mathscr{O}_{\mathbb{P}(\mathscr{E})})=1+y$. Furthermore, $\pi^{*}$ (and $\pi_{*}$) act on Chern-ext characters in the obvious manner. Now we apply $\pi_{*}$ to the equation above yielding

\[
\varphi_{*}\left(\text{ch}_{ext}(\Omega_{Y})\text{td}(Y)\right)=\pi_{*}\left(\frac{\text{ch}_{ext}(\mathscr{F}^{\vee})}{\text{ch}_{ext}(N^{\vee})(1+y)}\cdot \alpha\right)\text{ch}_{ext}(\Omega_{B})\text{td}(B)
\]
\\
by the projection formula. Thus computing $\varphi_{*}\left(\text{ch}_{ext}(\Omega_{Y})\text{td}(Y)\right)$ amounts to computing

\[
\pi_{*}\left(\frac{\text{ch}_{ext}(\mathscr{F}^{\vee})}{\text{ch}_{ext}(N^{\vee})(1+y)}\cdot \alpha\right).
\]

We spell out the details of this computation in the case of $D_5$ only, as the other cases differ inasmuch as the Chern roots of $\mathscr{F}$ and $N$ vary from case to case. For $D_5$ the Chern roots of $\mathscr{F}$ and $N$ are $(H,H+L,H+L,H+L)$ and $(2H+2L,2H+2L)$ respectively, where $H:=c_1(\mathscr{O}(1))$ and $L$ we will use to denote both $c_1(\mathscr{L})$ and $\pi^{*}c_{1}(\mathscr{L})$. Putting this all together we get

\[
\frac{\text{ch}_{ext}(\mathscr{F}^{\vee})}{\text{ch}_{ext}(N^{\vee})(1+y)}\cdot \alpha=\frac{(1+y\cdot e^{-H})(1+y\cdot e^{-H-L})^{3}}{(1+y\cdot e^{-2H-2L})^2(1+y)}\cdot \frac{H(H+L)^{3}(1-e^{-2H-2L})^{2}}{(1-e^{-H})(1-e^{-H-L})^3},
\]
\\
where we have cancelled a factor of $(2H+2L)^2$ from the numerator and denominator of $\alpha$. Now let $D=\frac{(1+y\cdot e^{-H})(1+y\cdot e^{-H-L})^{3}}{(1+y\cdot e^{-2H-2L})^2(1+y)}\cdot \frac{H(H+L)^{3}(1-e^{-2H-2L})^{2}}{(1-e^{-H})(1-e^{-H-L})^3}$. By the pushforward formula of \cite{James} we get

\begin{eqnarray*}
\pi_*(D)&=&\frac{1}{2}\frac{d^2}{dH^2}\left(\frac{D-(a_0+a_1H+a_2H^2)}{H}\right)\mid_{H=-L} \\
        &=&4-y + \frac{(y+1)(yU -3)}{(yU^2+1)} - \frac{U(y+1)^2}{(yU^2+1)^2}, \\
\end{eqnarray*}
where $U=e^{-L}$ and the $a_{i}$s are expressions in $L$ and $y$ obtained by expanding $D$ as a series in $H$. Identifying $\text{ch}_{ext}(\Omega_{Y})\text{td}(Y)$ and $\text{ch}_{ext}(\Omega_{B})\text{td}(B)$ with $\mathscr{H}_{y}(Y)$ and $\mathscr{H}_{y}(B)$ respectively yields the conclusion of Theorem ~\ref{Mainresult}. We note that reading off the coefficient of $y^q$ in the series $\varphi_{*}(\mathscr{H}_{y}(Y))$ gives us

\[
\varphi_{*}(\mathscr{H}_{q}(Y))=\sum_{i=0}^{q}P_{q-i}\mathscr{H}_{i}(B),
\]
where we recall that $\mathscr{H}_{q}(X):=\text{ch}(\Omega_{X}^{q})\text{td}(X)$ denotes the $q$th Hirzebruch characteristic class of a smooth variety $X$ and the $P_{i}$s are polynomials in $U=e^{-L}$. We list the explicit form of the $P_i$s for each case below:

\begin{center}
\begin{tabular}{|c|c|c|c|}
\hline
  & $P_0$ & $P_1$ & $P_n$ for $n>1$  \\
\hline
$D_5$ & $1-U$ & $2U^3+3U^2-U-4$ & $-U((n+1)U-n+2)(U-1)(U+1)^2(-U^2)^{n-2}$    \\
\hline
$E_6$ & $1-U$ & $U^4+2U^3+U^2-U-3$ & $-U^{2}(U^3-1)(U+1)^2(-U^3)^{n-2}$ \\
\hline
$E_7$ & $1-U$ & $U^5+U^4+U^3-U-2$ & $-U^3(U^4-1)(U^2+U+1)(-U^4)^{n-2}$ \\
\hline
$E_8$ & $1-U$ & $U^7+U^5-U-1$ & $-U^5(U^6-1)(U^2+1)(-U^6)^{n-2}$ \\
\hline

\end{tabular}
\end{center}

The fact that $P_0$ is the same in all cases is a consequence of the fact that $K_{Y}=\varphi^{*}(c_1(\mathscr{L})-c_1(B))$ for all the fibrations considered (see the appendix of \cite{EFY}). With the exception of the $D_5$ case, all roots of $P_{n}$ for $n>1$ lie on $(S^{1}\cup {0})\subset \mathbb{C}$ (for the $D_5$ case an ``anomolous" root of $\frac{n-2}{n+1}$ appears). We note that the length of our original proof was substantially greater, as we computed each $\chi_q$ individually and then proved a recursive relation between them. As it turned out, it was much easier to compute all of the $\chi_q$s at once, which we were easily able to do once armed with the Chern-ext character and a computer implementation of $\pi_{*}$ \cite{URL}.

\subsection{Proof of the corollary}

We first need some definitions. Let $R$ be a commutative ring with 1, let $f = a_0 + a_1 t + a_2 t^2 + \cdots \in R[[t]]$ and let $[t^d]: R[[t]] \rightarrow R$ be defined as in \S\ref{intro}.
If $\lambda_1,\ldots,\lambda_d \in R$, we use the notation $p_i := \lambda_1^i + \cdots \lambda_d^i$ and we let
\[ C := \prod_{i=1}^d (1 - \lambda_i t) = 1 - c_1 t + c_2 t^2 - c_3 t^3 + \cdots  \]
Then $c_i$ is the $i$th {\em symmetric polynomial}, and $p_i$ is the $i$th {\em power polynomial} of $\lambda_1,\ldots,\lambda_d$.
\begin{lemma}
\label{lemma1.1}
Let $f(t) = a_0 + a_1 t + \cdots \in R[[t]]$.  Then
\[ \sum_{i=1}^d f(\lambda_i t) = f \odot (d + p_1 t + p_2 t^2 + \cdots) = d a_0 + f \odot (-tC'/C), \]
where $\odot$ denotes the Hadamard product as defined in \S\ref{intro}.
\end{lemma}
\begin{proof}
The first equality follows from the definition of the $p_i$. For the second, note that $-tC'/C$ is well defined because the polynomial $C$
has a constant term of 1.  The equation $-tC'/C = p_1 t + p_2 t^2 + \cdots$ is obvious for $d=1$ (geometric series). For $d>1$, recall that logarithmic
derivatives turn products into sums: $(CD)'/(CD) = C'/C + D'/D$.
\end{proof}

Now let $Y$ be a $D_5$, $E_6$, $E_7$ or $E_8$ elliptic fibration over a base $B$ of some fixed dimension $d$. Then by Theorem~\ref{Mainresult} and equation ($\dagger$) in \S\ref{intro},

\[
\varphi_{*}\mathscr{H}_{y}(Y)=Q\cdot \mathscr{H}_{y}(B)=Q\cdot \prod_{i=1}^{d}g(\lambda_i),
\]
where the $\lambda_i$s are the Chern roots of the tangent bundle of $B$ and $g=(1+ye^{-t})\frac{t}{1-e^{-t}}$. Now let $f=\ln(g)=a_0+a_1t+\cdots$ (note that $a_{0}=\ln(1+y)$) . Lemma~\ref{lemma1.1} then yields

\begin{eqnarray*}
\mathscr{H}_{y}(B) =  \prod_{i=1}^d g(\lambda_i) &=& [t^d][ \prod_{i=1}^d g(\lambda_i t) ] \\
        &=&  [t^d][{\rm exp}( \sum_{i=1}^d f(\lambda_i t) ) ] \\
        &=&  [t^d][{\rm exp}( d a_0 + f \odot (-tC'/C)   ) ]\\
        &=&  (1+y)^d \cdot [t^d][{\rm exp}\left( f \odot (-tC'/C)   \right) ].
\end{eqnarray*}
The conclusion of Theorem~\ref{Mainresult} then states that over a base $B$ of dimension $d$ we have

\[
\tag{$\dagger \dagger$}
\varphi_{*}\mathscr{H}_{y}(Y)=Q\cdot  (1+y)^d \cdot [t^d][{\rm exp}\left( f \odot (-tC'/C)   \right)].
\]
But the right hand side of ($\dagger \dagger$) is just $[t^d]\overset \sim \chi(t,y)\mid_{t=t(1+y)}$, with $\overset \sim \chi(t,y)$ as defined in Corollary~\ref{genseries}. The corollary then follows.

\section{Discussion}

After going through the proof of the main result, one immediately notices that the only data needed from the elliptic fibrations under consideration were the Chern roots of its normal bundle in $\mathbb{P}(\mathscr{E})$ along with the Chern roots of the relative tangent bundle $T_{\mathbb{P}(\mathscr{E})/B}$ of the ambient projective bundle $\mathbb{P}(\mathscr{E})$. As such, our program can be carried out verbatim for any smooth subvariety of $\mathbb{P}(\mathscr{E})$, long as $\mathscr{E}$ is a direct summand of tensor powers of a fixed line bundle on the base (this assumption is needed to apply the pushforward formula from \cite{James}). More precisely, take any smooth complete intersection $X$ in some projective space $\mathbb{P}^n$ given by equations $X:(F_1=F_2=\cdots =F_m=0)$, promote the coefficients of the $F_i$ to appropriate sections of tensor powers of a fixed line bundle on some smooth base variety $B$ and we will have then constructed a fibration $\varphi:Y\to B$ such that the generic fiber is a complete intersection which is rationally equivalent to $X$. Then substituting the $n+1$ Chern roots of the relative tangent bundle of the ambient $\mathbb{P}^n$-bundle where $Y$ resides along with the $m$ Chern roots of the normal bundle to $Y$ into our calculations above will yield analogous results for the ``$X$ fibration" $\varphi:Y\to B$. Our results are thus genuinely more general than the title of this note suggests (see \cite{URL} for more details).

We conclude by noting that a true culmination of these results will not be achieved without a Lefschetz hyperplane type theorem for varieties in projective bundles, as hypersurfaces in projective bundles are almost never ample divisors (which is the key assumption of the Lefschetz hyperplane theorem). Once such a theorem is unveiled, only the middle cohomology will be unique to a hypersurface in a projective bundle thus rendering only $\lceil \frac{d}{2} \rceil$ of its Hodge numbers as non-trivial (where $d$ is the dimension of the hypersurface). The Hirzebruch invariants could then be used for the determination of the non-trivial Hodge numbers. As the cohomology of a projective bundle can be related to its base via the projective bundle theorem, a Lefschetz hyperplane type theorem along with the results in this paper would then relate all Hodge numbers of a hypersurface (and so complete intersections) in a projective bundle to invariants of the base.

\end{document}